\documentclass[12pt]{article}
\usepackage{amssymb, verbatim, amsmath,amsthm,amsfonts}
\usepackage{color}
\usepackage{algpseudocode}
\usepackage[all]{xy}
\usepackage{etex}
\usepackage[usenames,dvipsnames]{pstricks}
\usepackage{geometry}
\usepackage{fullpage}
\usepackage{amsthm}
\usepackage{MnSymbol}
\usepackage{enumitem}
\usepackage{pifont}
\usepackage{ifsym}
\usepackage{graphicx} 
\usepackage[hidelinks]{hyperref}
\parskip \baselineskip

\newcommand{\bc}{\begin{center}}

\newcommand{\finsub}[0]{{\subseteq}_{\text{\sf{fin}}}}
\newcommand{\ec}{\end{center}}
\newcommand{\be}{\begin{enumerate}}
\newcommand{\ee}{\end{enumerate}}
\newcommand{\beq}{\begin{equation}}
\newcommand{\eeq}{\end{equation}}
\newcommand{\bi}{\begin{itemize}}
\newcommand{\ei}{\end{itemize}}
\newcommand{\bd}{\begin{description}}
\newcommand{\ed}{\end{description}}
\newcommand{\ba}{\begin{array}}
\newcommand{\bea}{\begin{eqnarray*}}
\newcommand{\eea}{\end{eqnarray*}}
\newcommand{\ea}{\end{array}}
\newcommand{\bt}{\begin{tabular}}
\newcommand{\et}{\end{tabular}}

\newcommand{\mb}{\mbox}

\newcommand{\bmi}{\begin{minipage}}
\newcommand{\emi}{\end{minipage}}

\newcommand{\lb}{\linebreak}

\newtheorem{stel}{Theorem}[section]
\newtheorem{defin}[stel]{Definition}
\newtheorem{definlem}[stel]{Definition-Lemma}
\newtheorem{defintheo}[stel]{Definition-Theorem}
\newtheorem{lemm}[stel]{Lemma}

\newtheorem{rem}[stel]{Remark}

\newtheorem{theo}[stel]{Theorem}

\makeatletter
\newcommand{\myitem}[1]{%
\item[#1]\protected@edef\@currentlabel{#1}%
}
\makeatother
\usepackage{multicol}

\usepackage{MnSymbol,bbding,pifont}

\makeatletter
\newcommand{\Matrix}[1]
    {\begin{pmatrix}
      \Matrix@r #1;\@bye;\Matrix@r
     \end{pmatrix}}

\def\Matrix@r #1;{\@bye #1\Matrix@z\@bye\Matrix@s #1,\@bye, }%
\def\Matrix@s #1,{#1\Matrix@t }%
\def\Matrix@t #1,{\@bye #1\Matrix@y\@bye\@firstofone {&#1}\Matrix@t}%
\def\Matrix@y #1\Matrix@t{\\ \Matrix@r }%
\def\Matrix@z #1\Matrix@r {}
\def\@bye  #1\@bye   {}

\makeatother

\newcommand{\fr}[1]{\mathfrak{#1}}

\makeatother
\newcommand*\xoverline[2][0.75]{%
    \sbox{\myboxA}{$\m@th#2$}%
    \setbox\myboxB\null
    \ht\myboxB=\ht\myboxA%
    \dp\myboxB=\dp\myboxA%
    \wd\myboxB=#1\wd\myboxA
    \sbox\myboxB{$\m@th\overline{\copy\myboxB}$}
    \setlength\mylenA{\the\wd\myboxA}
    \addtolength\mylenA{-\the\wd\myboxB}%
    \ifdim\wd\myboxB<\wd\myboxA%
       \rlap{\hskip 0.5\mylenA\usebox\myboxB}{\usebox\myboxA}%
    \else
        \hskip -0.5\mylenA\rlap{\usebox\myboxA}{\hskip 0.5\mylenA\usebox\myboxB}%
    \fi}

\begin{document}

\bc {\bf\large Multiplicative near-vector spaces}\\[3mm]
{\sc L Boonzaaier, S. Marques and D. Moore} 

\it\small
rain (Pty) Ltd, 
Cape Quarter,
Somerset Road, 
Green Point, 8005,\lb
South Africa\\
\rm e-mail address: leandro.boonzaaier@rain.co.za

\it
Department of Mathematical Sciences, 
University of Stellenbosch, 
Stellenbosch, 7600,\lb
South Africa\\
\&
NITheCS (National Institute for Theoretical and Computational Sciences, 
South Africa \\
\rm e-mail: smarques@sun.ac.za

\it
Department of Mathematical Sciences, 
University of Stellenbosch, 
Stellenbosch, 7600,\lb
South Africa\\
\rm e-mail: dmoore@sun.ac.za
\ec

\begin{abstract} 

\noindent Near-vector spaces extend linear algebra tools to non-linear algebraic structures, enabling the study of non-linear problems. However, explicit constructions remain rare. This paper introduces a broad computable family of near-vector spaces, called multiplicative, and explores their properties. This family is fully determined over finite, real, and complex fields. We also discuss the existence of infinite coproducts, and products in the category of near-vector spaces. Finally, we introduce the complexification of a multiplicative near-vector space over the real numbers.
\end{abstract} 



\tableofcontents

\section{Introduction}
The notion of near-vector spaces was introduced by J.~André in \cite{Andre} as a generalization of vector spaces, where only one side of the distributive law is required to hold. This extension allows linear algebra tools to be applied to non-linear algebraic structures, facilitating the study of non-linear problems. It provides a framework for analyzing algebraic structures that exhibit non-linear behavior. Various aspects of near-vector spaces have been studied (see, for example, \cite{LSD2024, DeBruyn, HCW, HowMar, HowMey, HowRab, MarquesMoore, Dani}).

\noindent Near-linear algebra has the potential for real-world applications, tackling more complex problems than traditional linear algebra. However, without a broad range of concrete examples, these applications may remain out of reach. This paper addresses this challenge by introducing a large computable family of near-vector spaces, called multiplicative, which can already be explicitly constructed over finite fields, real fields, and complex fields, laying the groundwork for future applications. Real and complex vector spaces, along with those over finite fields, form a significant foundation of linear algebra. This family of near-vector spaces holds great promise in extending traditional near-linear algebra techniques.

\noindent With this family of near-vector spaces in hand, we can explore fundamental questions such as the existence of infinite products in the category of near-vector spaces. On one hand, we establish conditions on the scalar group $F$ ensuring that the infinite product of near-vector spaces remains a near-vector space over $F$ (see Theorem \ref{infprod} and Remark \ref{infprodrem}). On the other hand, we provide a counterexample demonstrating that this property does not always hold (see Section \ref{real}).

\noindent The paper is structured as follows:

\noindent In the next section, we introduce foundational material on the theory of near-vector spaces, establishing key definitions and results. The section concludes with a proof that the category of near-vector spaces admits infinite coproducts (see Definition-Lemma \ref{V^I}).

\noindent In section three, we recall the definition of a near-field induced by a multiplicative automorphism. We introduce the notion of a {multiplicative near-vector space}, a near-vector space constructed from families of multiplicative near-field automorphisms (see Definition \ref{multNVS}). We then compute the {quasi-kernel} for such spaces (see Theorem \ref{nvss}) and describe their {regular decomposition} (see Definition-Theorem \ref{pluses}). Additionally, we provide multiple characterizations of multiplicative near-vector spaces (see Theorem \ref{mult}). The section concludes with a proof that, under certain finiteness conditions, the infinite direct product of near-vector spaces remains a near-vector space (see Theorem \ref{infprod}).

\noindent In the last section, we describe multiplicative near-vector spaces over finite fields, real fields, and complex fields. We present an example of an infinite family of near-vector spaces that does not admit an infinite direct product. Finally, we discuss the {complexification of real multiplicative near-vector spaces}, extending a classical process from linear algebra to this more general setting.

\section{Preliminary material and notation}

\noindent In this section, we introduce preliminary material that is necessary to discuss the theory of multiplicative near-field structures in the section to follow. We also introduce notations which will be used throughout this paper. 
\begin{defin}\cite[Definition 1.1]{MarquesMoore} \label{scalargroup}
A {\sf scalar group} \( F \) is a tuple \( F = (F, \cdot, 1, 0, -1) \) where \((F, \cdot, 1)\) is a monoid, \(0, -1 \in F\), satisfying the following conditions: 
\begin{itemize}
  \vspace{-0.5cm}    \item \(0 \cdot \alpha = 0 = \alpha \cdot 0\) for all \(\alpha \in F\);
    \item \(\{ \pm 1 \}\) is the solution set of the equation \(x^2 = 1\);
    \item \((F \setminus \{ 0 \}, \cdot, 1)\) is a group.
\end{itemize}
 \vspace{-0.5cm} 
For all \(\alpha \in F\), we denote \(-\alpha\) as the element \((-1) \cdot \alpha\).
\end{defin}
\noindent Throughout this paper, we will use the notation \( F \) to represent the tuple \((F, \cdot, 1, 0, -1)\) when the context is clear. For this section, unless otherwise stated, $F$ denotes a scalar group.
\begin{defin}
    \cite[Definition 1.4]{MarquesMoore}
\begin{enumerate}    
 \vspace{-0.5cm}  \item Given a scalar group $F$, we say that a {\sf (left) $F$-space} is a pair $(V, \mu)$ where $V$ is an abelian additive group and $\mu: F \times V \rightarrow V$ is a (left) scalar group action, that is \(\mu\) is a (left) monoid action of $(F,\cdot)$ on $(V,+)$ on the (left) by endomorphisms, such that $0$ acts trivially, and $\pm 1$ act as $\pm {\sf id}$. 
\item We say that $W$ is an {\sf $F$-subspace} of $V$ if $W$ is a nonempty subset of $V$ that is closed under addition and scalar multiplication.
\end{enumerate}
\end{defin}
\noindent Next, we define the notion of the quasi-kernel of an $F$-space.
\begin{defin}\cite[Definition 1.5]{MarquesMoore}
\label{quasikernel}
Let $V$ be an $F$-space. We define the {\sf quasi-kernel} of $V$ to be
$$Q(V) = \{u \in V \mid \forall {\alpha, \beta \in F}, \exists {\gamma \in F} \,[\alpha \cdot u + \beta \cdot u = \gamma \cdot u]\}.$$
\end{defin}
\noindent We now introduce the concept of a (left) near-vector space over a scalar group. We adhere to the definition provided in \cite{MarquesMoore}.
\begin{defin}\cite[Definition 1.6]{MarquesMoore}
\label{NVS2}
A {\sf (left) near-vector space over a scalar group $F$} is a triple \((V, F, \mu)\) that satisfies the following properties:
\begin{enumerate}
 \vspace{-0.5cm} \item \((V, \mu)\) is an $F$-space; 
    \item the (left) scalar group action \(\mu\) is free;
    \item the quasi-kernel \(Q(V)\) generates \(V\) as an additive group.
\end{enumerate}
 \vspace{-0.5cm} 
Any trivial abelian group can be endowed with a near-vector space structure through the trivial action. Such a space is referred to as a {\sf (left) trivial near-vector space over $F$}, denoted as \(\{0\}\). We call the elements of $V$ {\sf vectors}.
\end{defin}

\noindent It was shown in \cite[Corollary 3.6]{MarquesMoore} that every near-vector space admits a basis contained in its quasi-kernel. Since these basis elements belong to the quasi-kernel, the usual linear algebraic definition of a basis applies in this setting. We use this fact to introduce the following definition.  

\begin{defin}  
     A near-vector space is said to be {\sf finite dimensional} if it admits a finite basis contained in its quasi-kernel.  
\end{defin}

\noindent In this paper, the focus will be on (left) near-vector spaces, although all results can be extended to (right) near-vector spaces as well. We will use the term ``subspace" for ``$F$-subspace" and ``near-vector space" instead of ``(left) near-vector space," unless otherwise specified. Additionally, we will adopt the notations \(V\) in place of \((V, \mu)\) (or $(V,F)$) when $F$ and $\mu$ are evident.

 \noindent In the following, \(F\) denotes a scalar group and \(V\) denotes a near-vector space over \(F\) unless mentioned otherwise. For a scalar group \(F\), we denote \(F \setminus \{0\}\) as \(F^{*}\), and for any subset \(W\) of \(V\) containing \(0\), we denote \(W \setminus \{0\}\) as \(W^{*}\).

\noindent When André established a set of axioms for a near-vector space \((V,F)\), he observed that defining an addition operation on \(F\) was not necessary. In fact, if the quasi-kernel of \(V\) contains more than just the element \(0\), for any non zero element of the quasi kernel, a natural addition on \(F\) can be introduced as follows.
\begin{defin} \label{+u} \cite[Definition 2.3]{Andre}
For ${u} \in Q(V)^{*}$ and $\alpha, \beta \in F$, we denote $\alpha +_{{u}} \beta$ as the unique element $\gamma$ of $F$ satisfying $\alpha {u} + \beta {u} = \gamma {u}$.
\end{defin}
\noindent Given $u \in Q(V)^{*}$ and a scalar group $F$, the operation $+_{u}$ forms a group operation on $F$. Furthermore, it was established in \cite[Satz 2.4]{Andre} that $(F,+_{u},\cdot)$ is a near-field. We denote $(F,+_{u},\cdot)$ as $F_{u}$.
 
\noindent In order to compare near-vector spaces, we need the following definition.
\begin{defin}{\cite[Definition 3.2]{HowMey}}\label{homo}
Let $((V_{1},+_{1}),\cdot_1)$ and $((V_{2},+_{2}),\cdot_2)$ be near-vector spaces over $(F_{1},*_{1})$ and over $(F_{2}, *_{2})$ respectively.
\begin{enumerate}
    \vspace{-0.5cm} 
    \item A pair \((\theta, \eta)\) is called a {\sf homomorphism} of near-vector spaces if  \(\theta : (V_1, +_{1}) \rightarrow (V_2, +_{2})\) is an additive homomorphism and  \(\eta : (F_1^{*}, *_{1}) \rightarrow (F_2^{*}, *_{2})\) is a monoid isomorphism such that \(\theta(\alpha \cdot_{1} u) = \eta(\alpha) \cdot_{2} \theta(u)\) for all \( u \in V_1 \) and \(\alpha \in F_1^{*}\). A pair \((\theta, \eta)\) is called a {\sf isomomorphism}  of near-vector spaces if \((\theta, \eta)\) is a homomorphism of near-vector spaces and \( \theta \) is an isomorphism. 
    \item If \(\eta = \operatorname{Id}\), we say that \(\theta\) is an \( F \)-{\sf linear map} where $F = F_{1} = F_{2}$. 
    \item If \(\eta = \operatorname{Id}\) and $\theta$ is a bijection, we say that \(\theta\) is an {\sf \( F \)-isomorphism} of near-vector spaces. 
\end{enumerate}
\vspace{-0.5cm} When \( ((V_{1},+_{1}),\cdot_1) \) and \( ((V_{2},+_{2}),\cdot_2) \) are vector spaces over the division rings \( (F_{1}, \boxplus_1, *_{1}) \) and \( (F_{2}, \boxplus_2 , *_{2}) \), respectively, definitions \( 1 \), \( 2 \), and \( 3 \) can be adapted to this setting. In this case, \( \eta \) would be a morphism of division rings. 
\end{defin}
\noindent We include the following immediate lemma, which will be used in the proof of Theorem \ref{infprod}.

\begin{lemm}  
\label{finitedimensional}  
Let $((V_{1},+_{1}),\cdot_1)$ and $((V_{2},+_{2}),\cdot_2)$ be near-vector spaces over $(F_{1},*_{1})$ and $(F_{2}, *_{2})$, respectively, such that  
\[
(((V_{1},+_{1}),\cdot_1),(F_{1},*_{1}))\simeq(((V_{2},+_{2}),\cdot_2), (F_{2}, *_{2}))
\]  
as near-vector spaces. Then, $((V_{1},+_{1}),\cdot_1)$ is finite-dimensional over $(F_{1},*_{1})$ if and only if $((V_{2},+_{2}),\cdot_2)$ is finite-dimensional over $(F_{2}, *_{2})$.  

A similar result holds when $((V_{1},+_{1}),\cdot_1)$ and $((V_{2},+_{2}),\cdot_2)$ are vector spaces over division rings $(F_{1},\boxplus_1, *_{1})$ and $(F_{2}, \boxplus_2, *_{2})$, respectively, and satisfy  
\[
(((V_{1},+_{1}),\cdot_1),(F_{1},\boxplus_1,*_{1}))\simeq(((V_{2},+_{2}),\cdot_2), (F_{2},\boxplus_2, *_{2}))
\]  
as vector spaces.  
\end{lemm}  

\begin{proof}
    By assumption, we know that there exists an isomorphism $(\psi, \eta)$ from $((V_{1},+_{1}),\cdot_1)$ to $((V_{2},+_{2}),\cdot_2)$. Suppose that $((V_{1},+_{1}),\cdot_1)$ is finite-dimensional over $(F_{1},*_{1})$. Then, we can choose a basis $\mathcal{B}$ for $((V_{1},+_{1}),\cdot_1)$ such that $\mathcal{B} \subseteq Q(V)^{*}$ and $|\mathcal{B}| < \infty$.  

    We aim to show that $\psi(\mathcal{B})$ is a basis for $((V_2,+_2),\cdot_2)$ over $(F_2,*_2)$.  

    First, we prove that $\psi(\mathcal{B})$ generates $V_{2}$. Take any $w \in V_{2}$. Since $\psi$ is surjective, there exists $v \in V_{1}$ such that $w = \psi(v)$. Since $\mathcal{B}$ is a basis for $V_1$, we can express $v$ as  
    \[
    v = {}^{1}\!\sum_{b \in \mathcal{B}}\lambda_{b} \cdot_1 b, \quad \text{for some } \lambda_{b} \in F_{1}, \forall b \in \mathcal{B}.
    \]
    Applying $\psi$ and using that $(\psi, \eta)$ is a homorphism of near-vector space, we obtain  
    \[
    w = \psi(v) = {}^{2}\!\sum_{b \in \mathcal{B}}\eta(\lambda_{b})\cdot_{2}\psi(b),
    \]
    which shows that $\psi(\mathcal{B})$ spans $V_{2}$.  

    Now, we prove linear independence. Suppose that there exist $\gamma_{b} \in F_{2}$ for all $b \in \mathcal{B}$ such that  
    \begin{equation}
    \label{first}
       {}^{2}\!\sum_{b \in \mathcal{B}}\gamma_{b} \cdot_2\psi( b) = 0.
    \end{equation}
    Since $\eta$ is an isomorphism, for each $b \in \mathcal{B}$, there exists $\lambda_{b} \in F_{1}$ such that $\gamma_{b} = \eta(\lambda_{b})$. Substituting into (\ref{first}), we get  
    \[
    {}^{2}\!\sum_{b \in \mathcal{B}}\eta(\lambda_{b})\cdot_{2}\psi(b) = 0.
    \]
    Since $(\psi, \eta)$ is an isomorphism of near-vector spaces, this implies  
    \[
    \psi\left({}^{1}\!\sum_{b \in \mathcal{B}}\lambda_{b}b\right) = 0.
    \]
    Applying injectivity of $\psi$, we obtain  
    \[
    {}^{1}\!\sum_{b \in \mathcal{B}}\lambda_{b}b = 0.
    \]
    Since $\mathcal{B}$ is a basis for $((V_1,+_1),\cdot_1)$ over $(F_1,*_1)$, it follows that $\lambda_{b} = 0$ for all $b \in \mathcal{B}$. Hence, $\gamma_{b} = \eta(\lambda_{b}) = 0$, and so $\lambda_b=0$ for all $b \in \mathcal{B}$, proving linear independence of $\psi(\mathcal{B})$.  

    Therefore, $\psi(\mathcal{B})$ is a basis for $((V_2,+_2),\cdot_2)$ over $(F_2,*_2)$.  

    The converse follows similarly by applying $(\psi^{-1}, \eta^{-1})$ instead of $(\psi, \eta)$.  
\end{proof}

\noindent The following lemma will be used in the proof of Definition-Lemma \ref{V^I} and Theorem \ref{mult}.

\begin{lemm}\label{homoplus}
    Let $V_1$ and  $V_2$ be $F$-spaces and $\phi$ be an $F$-linear map of near-vector spaces between $V_1$ and $V_2$. Then for all $u \in Q(V_{1})^{*}$, $+_{u} = +_{\phi(u)}$.
\end{lemm}
\vspace{-0.5cm}
\begin{proof}
    Let $\alpha, \beta \in F$ and $u \in Q(V_1)^{*}$. Then, 
    \begin{align*}
        \alpha \cdot_{2}\phi(u)+_{2} \beta \cdot_{2}\phi(u) & = \phi (\alpha \cdot_{1}u+_{1}\beta \cdot_{1}u) = \phi((\alpha+_{u}\beta)\cdot_{1}u) = (\alpha +_{u}\beta)\cdot_{2}\phi(u). 
    \end{align*}
    Since we have proven the above sequence of equalities for all $\alpha, \beta \in F$, we have $+_u = +_{\phi(u)}$.
\end{proof}

\noindent The following definition will be needed to describe an infinite coproduct in the category of near-vector spaces (see Definition-Lemma \ref{V^I}).
\begin{defin}
Let $\{A_i\}_{i\in I}$ be a family of abelian groups,  and $(a_i)_{i\in I} \in \prod_{i\in I} A_i$. We say that the $a_i$'s are almost all zero if there are non-zero values for only a finite number of indices. We denote $(a_i)_{i\in I}\asymp \mathbf{0} $ to symbolize this property.
\end{defin}
\noindent In the following definition, we introduce the generalized Kronecker delta, which will be used to define the standard basis for a multiplicative near-vector space in the next section.
\begin{defin}
Let $I$ be a set. We define the generalized Kronecker delta as the symbol $\delta_{i,j}$, which is $0$ when $i\neq j$ and $1$ otherwise.
\end{defin}
\noindent In the following definition-lemma, we establish the existence of infinite coproducts in the category of near-vector spaces. While this result is likely well-known, we have not found an explicit proof in the literature. Therefore, for completeness, we provide a quick proof demonstrating that an infinite coproduct of near-vector spaces indeed forms a near-vector space.
\begin{definlem}
\label{V^I}
Let $I$ be an index set, and $\{((V_i,+_i), \cdot_i) \}_{i\in I}$ be a family of near-vector spaces over $F$.  
\begin{enumerate}
    \item We define the near-vector space denoted by $\bigoplus_{i \in I}V_{i}$, to be the set 
$$\bigoplus_{i \in I}V_{i}=\left\{ (v_i)_{i\in I} \in \prod_{i\in I} V_i \mid (v_i)_{i\in I} \asymp \mathbf{0} \right\},$$
endowed with the addition and scalar multiplication component-wise. For any $i \in I$, we define the $F$-linear map $\iota_i: V_{i} \rightarrow \bigoplus_{i \in I}V_{i}$ by sending $v_i$ to $(v_i\delta_{i,j})_{j \in I}$.  
\item $(\bigoplus_{i \in I} V_{i}, \{\iota_{i}:V_{i}\rightarrow \bigoplus_{i \in I}V_{i}\}_{i \in I})$ is an infinite coproduct for the family $\{((V_i,+_i), \cdot_i) \}_{i\in I}$ in the category of near-vector spaces.
\end{enumerate}
\end{definlem}
\begin{proof}
\begin{enumerate}
    \item 
We show that $Q(\bigoplus_{i \in I}V_{i})$ generates $\bigoplus_{i \in I}V_{i}$ as an additive group. By Lemma \ref{homoplus}, we know that $\bigcup_{i \in I}\iota_{i}(Q(V_{i})) \subseteq Q(\bigoplus_{i \in I}V_{i})$, since $\iota_i$ is an $F$-linear map. Let $(v_{i})_{i \in I} \in \bigoplus_{i \in I}V_{i}$. We have $$(v_{i})_{i \in I} = \sum_{i \in I}v_{i}(\delta_{j,i})_{i \in I} = \sum_{i \in I} \left( \sum_{j \in J} u_{j_{i}}\right) (\delta_{j,i})_{i \in I} = \sum_{i \in I}\sum_{j \in J}u_{j_i}(\delta_{j,i})_{i \in I}$$ 
for some $u_{j_{i}} \in Q(V_{i})$ for all $i,j \in I$, since each $V_i$ is generated by $Q(V_i)$.  Since 
$u_{j_i}(\delta_{j,i})_{i \in I} \in Q(\bigoplus_{i \in I}V_{i})$ for all $i,j \in I$, this shows that $Q(\bigoplus_{i \in I}V_{i})$ generates $\bigoplus_{i \in I}V_{i}$ as an additive group. From there, it is not hard to prove that $\bigoplus_{i \in I}V_{i}$ is a near-vector space.
\item This is clear.
\end{enumerate}
\end{proof}
\vspace{-0.5cm}
\noindent We recall the definition of regularity, a fundamental concept in the study of near-vector spaces.
\begin{defin}{\cite[Definition 4.7]{Andre}} \label{def4} 
Let \( u, v \in Q(V)^{*} \).
\begin{enumerate}
    \vspace{-0.5cm}  \item We say that \( u \) and \( v \) are {\sf compatible} if there exists \( \lambda \in F^{*} \) such that \( u + \lambda v \in Q(V) \).
    \item We say that \( V \) is {\sf regular} if any two vectors of \( Q(V)^{*} \) are compatible.
\end{enumerate}
\end{defin}
\vspace{-0.5cm}
\noindent Regular near-vector spaces behave most like traditional vector spaces. It is straightforward to see that if $Q(V) = V,$ then $V$ is regular. Andr\'{e} proved that any near-vector space can be decomposed into regular parts (see \cite[Satz 4.13]{Andre}). This serves as motivation for him to refer to regular near-vector spaces as the building blocks of his theory. 

\noindent We end this section with the definition of a regular decomposition family for near-vector spaces.
\begin{defin}\cite[Definition 2.6]{LSD2024}
\label{decompfamily} 
Let \( v \in V \), \(\{V_{i}\}_{i \in I}\) be a family of subspaces of \( V \) and \(\{v_{i}\}_{i \in I}\) be a family of elements in \( V_{i} \) for each \( i \in I \).
\begin{enumerate}
   \vspace{-0.5cm}   \item We say that \(\{V_{i}\}_{i \in I}\) is a {\sf regular decomposition family for \( V \)} if \( V = \bigoplus_{i \in I} V_{i} \) where \( V_{i} \) is a maximal regular subspace of \( V \), for all $i\in I$.
    \item We say that \(\{v_{i}\}_{i \in I}\) is the {\sf family of regular components of \( v \) along \(\{V_{i}\}_{i \in I}\)} if \( v = \sum_{i \in I} v_{i} \) and \( v_{i} \in V_{i} \) for all \( i \in I \), where \(\{V_{i}\}_{i \in I}\) is a regular decomposition family for \( V \).
\end{enumerate}
\end{defin}

\section{Multiplicative near-field structures}
We begin this section by introducing the definition of a (left) near-field that will serve as the foundation for this paper. 
\begin{defin}
    A {\sf (left) near-field} $F$ is a tuple $(F,+,\cdot, 0,1)$ where $(F,\cdot,1)$ is a monoid, $(F^{*},\cdot,1)$ is a group, $(F,+,0)$ is an abelian group and for all $\alpha,\beta,\gamma \in F$, we have $\alpha (\beta +\gamma) = \alpha \beta +\alpha \gamma$. We also denote a near-field by $(F, +, \cdot)$, or simply $F$ for brevity.
\end{defin} 
\noindent We also recall the definition of a right-distributive element in a near-field.
\begin{defin}\cite[Definition 3.2]{LSD2024}
    Let $(F, +, \cdot)$ be a near-field. An element \(\gamma \in F\) is {\sf right distributive} if, for every \(\alpha, \beta \in F\), \((\alpha + \beta)\gamma = \alpha \gamma + \beta \gamma\). We denote by \(F_{\fr{d}}\) the set of all right distributive elements of \(F\).
\end{defin}\vspace{-0.5cm}
\noindent Given a near-field $(F,+,\cdot)$, it is shown in \cite[Theorem-2.4-5]{DeBruyn} that $(F,+,\cdot)$ is a division ring and $F$ is a (right) vector space over $F_{\fr{d}}$.

\noindent In this section, unless stated otherwise, \( F \) denotes a (left) near-field. For clarity, throughout the remainder of this paper, we will use the term "near-field" to refer specifically to a "(left) near-field."  

\noindent The primary objective of this paper is to construct a broad class of computable near-vector spaces by leveraging multiplicative automorphisms. To lay the groundwork for this construction, we first introduce the concept of a multiplicative automorphism of a near-field.
\begin{defin} \label{multauto}  
A {\sf multiplicative automorphism $\sigma$ of the near-field $(F, +, \cdot)$} is defined as a monoid automorphism of $(F, \cdot)$. In other words, $\sigma$ satisfies the following properties:  
\begin{itemize} 
 \vspace{-0.5cm}  \item $\sigma(1) = 1$, and  
\item $\sigma(\alpha \beta) = \sigma(\alpha) \sigma(\beta)$ for all $\alpha, \beta \in F$.
\end{itemize}
 \vspace{-0.7cm} 
\end{defin}
\noindent For the remainder of this paper, we shall use the term ``multiplicative automorphism" in place of ``multiplicative near-field automorphism" or ``multiplicative automorphism of a near-field" unless otherwise stated.
\noindent It is important to note that any multiplicative automorphism $\sigma$ of a near-field $(F,+,\cdot)$ has the following properties: 
\begin{enumerate}
    \vspace{-0.5cm}  \item $\sigma(0) = 0$; 
    \item $  \sigma(-\alpha)= -\sigma(\alpha)$ for all $\alpha \in F$;
    \item $\sigma(\alpha^{-1}) = (\sigma(\alpha))^{-1}$ for all $\alpha \in F^{*}$.
\end{enumerate}  \vspace{-0.5cm} 
1. This follows from the equality $\sigma(0)\sigma(0)= \sigma(0)$ and the following facts: $\sigma(0)\neq 1$, $\sigma$ is a bijection and any non-zero element has an inverse in $F$. \\
2. This follows clearly when $\alpha=0$ and in characteristic $2$. When $\alpha\in F^*$ and $\operatorname{char}(F)\neq 2$, the result follows from the equality 
\(\sigma(\alpha)^2 - \sigma(-\alpha)^2 = 0,
\)
which is obtained from the fact that \(\sigma\) is a multiplicative morphism. Since $\sigma$ is a bijection, we know $\sigma (\alpha) \neq \sigma(-\alpha)$ which then implies that $\sigma(-\alpha) = -\sigma(\alpha)$. \\
3. This follows from the equality $\sigma(\alpha)\sigma(\alpha^{-1})=\sigma(\alpha^{-1})\sigma(\alpha) = 1$, which is obtained using the fact that $\sigma$ is a multiplicative automorphism and every non-zero element of $F$ has a multiplicative inverse. 
\noindent Next, given a near-field $(F, +, \cdot)$, we define an addition and a scalar multiplication associated with a multiplicative automorphism. Using these definitions, we construct a near-field induced by the multiplicative automorphism. The proofs of the statements in this definition-lemma are straightforward and can be verified by explicitly writing out the required equations; therefore, we omit them for brevity.
\begin{definlem}  
\label{induced}  
Let $\sigma$ and $\rho$ be multiplicative automorphisms of $F$.  
\begin{enumerate}  
     \vspace{-0.5cm}  \item We define the \textsf{addition associated with $\sigma$} as a binary operation $+_\sigma$ on $F$, for all $\alpha, \beta \in F$, as  
    \[
    \alpha +_\sigma \beta = \sigma^{-1}(\sigma(\alpha) + \sigma(\beta)).
    \]  
    The triple $(F, +_\sigma, 0)$ forms an abelian group.  
    If the near-field is denoted by $F_u$ for some $u \in Q(V)^*$, this addition is written as $+_{u, \sigma}$.     Additionally, for $n \in \mathbb{N}$, we define  
    \[
    {}^{\sigma \! \!}\sum_{k=1}^{n} \alpha_k = \alpha_1 +_\sigma \alpha_2 +_\sigma \cdots +_\sigma \alpha_n.
    \]
  
 \item We define the \textsf{near-field induced by $\sigma$}, denoted by $F_{\sigma}$, as the near-field $(F, +_{\sigma}, \cdot)$. 
 
  \item We define the \textsf{scalar multiplication induced by $\rho$} as the scalar multiplication $\cdot_\rho$ on $F$, for all $\alpha, \beta \in F$, as 
    \[
    \alpha \cdot_\rho \beta = \rho(\alpha) \cdot \beta.
    \]  
    This operation $\cdot_\rho$ defines a left monoid action of $F$ on itself.  
\end{enumerate}    
\end{definlem}
\noindent We prove the next lemma that will be used in the proof of Theorem \ref{infprod}
\begin{lemm}
\label{divisionring}
   Let $\sigma$ be a near-field multiplicative automorphism. Then:  
    \vspace{-0.5cm}
    \begin{enumerate}
        \item The structure $(\sigma(F_{\fr{d}}),+_{\sigma^{-1}},\cdot)$ is isomorphic to $(F_{\fr{d}},+, \cdot)$ as division rings via $\sigma^{-1}$.
        \item $((F,+_{\sigma^{-1}}), \cdot)$ is a vector space over $(\sigma(F_{\fr{d}}),+_{\sigma^{-1}},\cdot)$, which is isomorphic to $((F,+),\cdot)$ over $(F_{\fr{d}},+,\cdot)$ via $(\sigma^{-1},\sigma^{-1})$ as vector spaces.
    \end{enumerate}
\end{lemm}

\begin{proof}
\begin{enumerate}
    \item Let $\alpha, \beta \in \sigma(F_{\fr{d}})$. Then, we have:
    \begin{align*}
        \sigma^{-1}(\alpha+_{\sigma^{-1}} \beta) &= \sigma^{-1}(\sigma(\sigma^{-1}(\alpha)+\sigma^{-1}(\beta))) \\
        &= \sigma^{-1}(\alpha) + \sigma^{-1}(\beta),
    \end{align*}
    and  
    \begin{align*}
        \sigma^{-1}(\alpha\beta) = \sigma^{-1}(\alpha) \cdot \sigma^{-1}(\beta),
    \end{align*}
    which shows that $\sigma^{-1}$ is a homomorphism. Since $\sigma$ is bijective by assumption, the result follows.

    \item The proof follows similarly to part (1).  
\end{enumerate}
\end{proof}

\noindent Given a near-vector space \( V \) over  \( F \), any \( u \in Q(V)^* \) defines an addition \( +_u \) (see \cite[Section 2]{Andre} or \cite[Definition 2.2-4]{DeBruyn}). For any \( \gamma \in F^* \), \( \gamma u \in Q(V)^* \) (see \cite[Hilfssatz 2.2(c)]{Andre} or \cite[Lemma 2.2-2(c)]{DeBruyn}), and the addition \( +_{\gamma u} \) is defined for all \( \alpha, \beta \in F \) as  
\[
\alpha +_{\gamma u} \beta = (\alpha \gamma +_{u} \beta \gamma) \gamma^{-1}
\]  (see \cite[Satz 2.5]{Andre} or \cite[Theorem 2.2-8]{DeBruyn}).

\noindent We can express this addition as an addition associated with a multiplicative automorphism using the following definition.
\begin{defin}
Let \( \gamma \in F^* \). The \textsf{multiplicative automorphism of \( F \) associated to \( \gamma \)}, denoted by \( \varphi_\gamma \), is defined as the map sending \( \alpha \) to \( \gamma^{-1} \alpha \gamma \). 
\end{defin}  \vspace{-0.5cm} 
\begin{rem}
If \( F \) is a field, then for every \( \gamma \in F \), the morphism \( \varphi_\gamma \) is the identity map.
\end{rem}\vspace{-0.5cm} 
\noindent We now obtain the following lemma:
\begin{lemm}
\label{varphi}
Given a near-vector space \( V \) over a scalar group \( F \), with \( u \in Q(V)^* \) and \( \gamma \in F^* \), we have 
$$+_{\gamma u}= +_{u,\varphi_{\gamma}}.$$
\end{lemm}  \vspace{-0.5cm} 
\begin{proof} 
    Let \( \alpha, \beta \in F \). Then:
    \begin{align*}
        \alpha +_{\gamma u} \beta &= (\alpha \gamma +_u \beta \gamma)\gamma^{-1}  \\
        &= (\gamma \gamma^{-1})(\alpha \gamma +_u \beta \gamma)\gamma^{-1} \\
        &= \gamma (\gamma^{-1} \alpha \gamma +_u \gamma^{-1} \beta \gamma)\gamma^{-1} \\
        &= \varphi_{\gamma}^{-1}(\varphi_{\gamma}(\alpha) +_u \varphi_{\gamma}(\beta)) \\
        &= \alpha +_{u,\varphi_{\gamma}} \beta.
    \end{align*}
    This follows from the left distributive property of the near-field \( F \).
\end{proof}
\vspace{-0.5cm}
\noindent The following lemma characterizes the condition under which the additions associated with two multiplicative automorphisms of a near-field are equal:
\begin{lemm}
\cite[Lemma 1.6]{BM2024}
\label{nearfieldauto}
    Let \( \sigma \) and \( \sigma' \) be multiplicative automorphisms of \( F \). Then, \( +_{\sigma} = +_{\sigma'} \) if and only if \( \sigma \circ \sigma'^{-1} \) is a near-field automorphism.
\end{lemm}  \vspace{-0.5cm} 
\noindent In the following result, we define one-dimensional near-vector spaces, which serve as the building blocks of the near-vector spaces discussed in this paper.
\begin{definlem}
\label{multnvs}
Let $\sigma$ and $\rho$ be two multiplicative automorphisms. We define the {\sf elementary multiplicative near-vector space associated with $(\sigma , \rho)$ over $(F, \cdot)$}, denoted by $F^{\sigma,\rho}$, as the near-vector space $((F, +_{\sigma}), \cdot_{\rho})$. We have $Q(F^{\sigma,\rho}) = F$, and for all $\gamma \in F^{*}$, we have 
$$+_{\gamma} = +_{\sigma(\gamma), \sigma \circ \rho} = +_{\varphi_{\sigma(\gamma)} \circ \sigma \circ \rho} = +_{\sigma \circ \varphi_{\gamma} \circ \rho}.$$ In particular, when $\gamma = 1$, we have $+_1 = +_{\sigma \circ \rho}$. In this paper, there will be no ambiguity regarding which underlying near-vector space is being referenced when using \(+_\gamma\). However, if clarification were ever needed, a superscript could be added to explicitly indicate that we are working within the near-vector space \(F^{\sigma,\rho}\), writing \(+_\gamma^{\sigma , \rho}\) instead.
\end{definlem}
\begin{proof} 
We begin by proving that \( F^{\sigma,\rho} \) is a near-vector space. The left distributivity follows from the fact that \(\sigma\) and \(\rho\) are multiplicative automorphisms, and the left distributivity property for the near-field \((F, + , \cdot)\). Furthermore, \(0\), \(1\), and \(-1\) act as the trivial morphism, the identity morphism, and the negative morphism respectively, since \(\rho\) is a multiplicative automorphism, and we know that \(\sigma(0) = 0\), and \(\sigma(\pm 1) = \pm 1\) as shown earlier in this section.
\noindent To complete the proof that \( ((F, +_{\sigma}), \cdot_{\rho}) \) is a near-vector space, we now show that \( Q(F^{\sigma,\rho}) = F \). Let \( \gamma \in F^* \) and \( \alpha, \beta \in F \). We have the following:
  \begin{align*}
        \alpha \cdot_{\rho}\gamma +_{\sigma}\beta \cdot_{\rho}\gamma & = \sigma^{-1}(\sigma(\rho(\alpha)\gamma)+\sigma(\rho(\beta)\gamma)) \\
        & = \sigma^{-1}(\sigma(\rho(\alpha))\sigma(\gamma)+\sigma(\rho(\beta))\sigma(\gamma)) \\
        & = \sigma^{-1}(\sigma(\rho(\alpha))+_{\sigma(\gamma)} \sigma(\rho(\beta)))\sigma^{-1}(\sigma(\gamma)) \\
        & = \sigma^{-1}(\sigma(\rho(\alpha))+_{\sigma(\gamma)} \sigma(\rho(\beta)))\gamma \\
        & = \rho^{-1}(\sigma^{-1}(\sigma(\rho(\alpha))+_{\sigma(\gamma)} \sigma(\rho(\beta))))\cdot_{\rho}\gamma \\
        & = (\alpha +_{\sigma(\gamma), \sigma \circ \rho}\beta)\cdot_{\rho}\gamma.
    \end{align*}   
\noindent This shows that \( \gamma \in Q(F^{\sigma,\rho}) \), and that \( +_\gamma = +_{\sigma(\gamma), \sigma \circ \rho} \). This completes the proof that \( Q(F^{\sigma,\rho}) = F \).\\
On the other hand, we have that: 
   \begin{align*}
       (\alpha +_{\varphi_{\sigma(\gamma)}\circ\sigma \circ \rho}\beta)\cdot_{\rho}\gamma & = \sigma^{-1}(\sigma(\gamma)(\sigma(\gamma)^{-1}\sigma(\rho(\alpha))\sigma(\gamma)+ \sigma(\gamma)^{-1}\sigma(\rho(\beta))\sigma(\gamma))\sigma(\gamma)^{-1})\gamma \\
       & = \sigma^{-1}(\sigma(\rho(\alpha))+_{ \sigma(\gamma)} \sigma(\rho(\beta)))\gamma \\
       & = \rho^{-1}(\sigma^{-1}(\sigma(\rho(\alpha))+_{ \sigma(\gamma)} \sigma(\rho(\beta))))\cdot_{\rho}\gamma \\
       & = (\alpha+_{\sigma(\gamma),\sigma\circ\rho}\beta)\cdot_{\rho}\gamma.
    \end{align*}  
Thus, we have \( +_{\varphi_{\sigma(\gamma)}\circ \sigma \circ \rho} = +_{\sigma(\gamma),\sigma\circ \rho} \).\\
Finally, we have:
  \begin{align*}
     (\alpha +_{ \sigma\circ \varphi_{\gamma} \circ\rho }\beta)\cdot_{\rho}\gamma & =  \rho^{-1}(\varphi_{\gamma}^{-1}(\sigma^{-1}(\sigma(\varphi_{\gamma}(\rho(\alpha)))+ \sigma(\varphi_{\gamma}(\rho(\beta))))))\cdot_{\rho}\gamma \\
     & = \gamma \sigma^{-1}(\sigma(\gamma^{-1}\rho(\alpha)\gamma)+\sigma(\gamma^{-1}\rho(\beta)\gamma)) \\
      &=  \alpha \cdot_{\rho}\gamma +_{\sigma}\beta \cdot_{\rho}\gamma
.    \end{align*} 
Thus, we conclude that \( +_{  \sigma  \circ \varphi_{\gamma} \circ \rho } = +_{\gamma} \), as desired.
\end{proof}
\noindent Next, we introduce and investigate the structure of a multiplicative near-vector space, built from families of multiplicative automorphisms.
\begin{definlem}
\label{multautnvs}
Let \((F, +, \cdot)\) be a near-field, and let \(I\) be an index set. Consider two tuples of multiplicative automorphisms, \(\boldsymbol{\sigma} = (\sigma_i)_{i \in I}\) and \(\boldsymbol{\rho} = (\rho_i)_{i \in I}\).  
\begin{enumerate} 
     \vspace{-0.5cm} \item We define the {\sf multiplicative near-vector space associated with \((\boldsymbol{\sigma}, \boldsymbol{\rho})\)}, denoted by \(F^{\boldsymbol{\sigma}, \boldsymbol{\rho}}\), as the near-vector space  
    \[
    F^{\boldsymbol{\sigma}, \boldsymbol{\rho}} = \left\{ (\alpha_i)_{i\in I} \in \prod_{i\in I} F^{\sigma_i, \rho_i} \mid (\alpha_i)_{i\in I} \asymp \mathbf{0} \right\}.
    \]  
    where the near vector space structure is given componentwise by $F^{\sigma_i, \rho_i}$, for all $i \in I$. Additionally, for \( k \in \{1, \dots, n\} \), we define  
\[
 {}^{\boldsymbol{\sigma} \! \! }\sum_{k=1}^{n}(\alpha_{k_i})_{i \in I} = \left( {}^{\sigma_{i}\! \! }\sum_{k=1}^{n} \alpha_{k_i}\right)_{i \in I},
\]  
and we adopt the convention \( F^{0,0} = \{0\} \).
    \item The {\sf canonical basis of \(F^{\boldsymbol{\sigma}, \boldsymbol{\rho}}\)} is given by the set \(\mathcal{E} = \{ {\bf e}_j \mid j \in I\}\), where \({\bf e}_j = (\delta_{j,i})_{i\in I}\). For all $j \in I$ and $\gamma \in F^{*}$, we have
    $$+_{\gamma \cdot_{\boldsymbol{\rho}} {\bf e}_j}=+_{\sigma_{j} \circ \rho_{j}\circ \varphi_{\gamma}  }.$$
\end{enumerate}
\end{definlem}\vspace{-0.5cm}
\begin{proof} By Lemma \ref{V^I}, we know that \(F^{\boldsymbol{\sigma}, \boldsymbol{\rho}}\) is a near-vector space over \(F\).
We demonstrate that \( Q(F^{\boldsymbol{\sigma}, \boldsymbol{\rho}}) \) generates \( F^{\boldsymbol{\sigma}, \boldsymbol{\rho}} \).  
Let \( j\in I \), and take \( \alpha, \beta \in F \) and \( \gamma \in F^* \). Then, we have:   
\begin{align*}
    &\alpha \cdot_{\boldsymbol{\rho}}(\gamma \cdot_{\boldsymbol{\rho}} {\bf e}_{j}) +_{\boldsymbol{\sigma}} \beta \cdot_{\boldsymbol{\rho}}(\gamma \cdot_{\boldsymbol{\rho}} {\bf e}_{j})\\
    &= (\alpha \gamma) \cdot_{\boldsymbol{\rho}} {\bf e}_{j} +_{\boldsymbol{\sigma}} (\beta \gamma) \cdot_{\boldsymbol{\rho}} {\bf e}_{j}  \\
    &= (\sigma_{i}^{-1}(\sigma_{i}(\rho_{i}(\alpha \gamma)\delta_{j,i}) + \sigma_{i}(\rho_{i}(\beta\gamma)\delta_{j,i})))_{i \in I} \\
    & = (\rho_{i}^{-1}(\sigma_{i}^{-1}(\sigma_{i}(\rho_{i}(\alpha \gamma)) + \sigma_{i}(\rho_{i}(\beta\gamma))))\cdot_{\rho_{i}}\delta_{j,i})_{i \in I} \\
    &= ((\varphi_{\gamma}^{-1}(\rho_{i}^{-1}(\sigma_{i}^{-1}(\sigma_{i}(\rho_{i}(\varphi_{\gamma}(\alpha))\delta_{j,i}) + \sigma_{i}(\rho_{i}(\varphi_{\gamma}(\beta))\delta_{j,i})) ))\gamma)\cdot_{\rho_{i}}\delta_{j,i})_{i \in I} \\
    &= (((\alpha+_{\sigma_{i}\circ\rho_{i}\circ\varphi_{\gamma}}\beta)\gamma) \cdot_{\rho_{i}}\delta_{j,i})_{i \in I} \\
    & = (\alpha +_{\sigma_{j} \circ \rho_{j} \circ \varphi_{\gamma}} \beta) \cdot_{\boldsymbol{\rho}}(\gamma \cdot_{\boldsymbol{\rho}} {\bf e}_{j}).
\end{align*}  
This confirms that \( +_{\gamma \cdot_{\boldsymbol{\rho}}{\bf e}_{j}} = +_{\sigma_{j}  \circ \rho_{j}\circ \varphi_{\gamma}} \).  \\
\end{proof}
\vspace{-0.5cm}\noindent In the following lemma, we describe the quasi-kernel of $F^{\boldsymbol{\sigma}, \boldsymbol{\rho}}$.

\begin{theo}\label{nvss}
    Let $(F,+,\cdot)$ be a near-field, $I$ be an index set and $\boldsymbol{\sigma} = (\sigma_i )_{i\in I}$, $\boldsymbol{\rho}=( \rho_i )_{i\in I}$ be tuples of multiplicative automorphisms. We define an equivalence relation $\sim$ on $I$ as follows: for $i,j \in I$, $i \sim j$ if $\sigma_{i} \circ \rho_{i} \circ \rho_{j}^{-1}\circ \sigma_{j}^{-1}$ is a near-field automorphism. Let $\mathcal{S}$ be the full set of representatives with respect to the equivalence relation $\sim$ and for any $s \in \mathcal{S}$, we denote $I_{s}$ to be the equivalence class of $s$ for this equivalence relation. Let $V_{s} = \{ (\alpha_{i})_{i \in I}\in F^{\boldsymbol{\sigma}, \boldsymbol{\rho}} \mid \alpha_{i} = 0 \text{ for all } i \notin I_{s}\} $. Then, we have \begin{equation*}
        Q(F^{\boldsymbol{\sigma}, \boldsymbol{\rho}}) = \bigcup_{s\in \mathcal{S}}\left\{\lambda \cdot_{\boldsymbol{\rho}} (k_{i})_{i \in I} \mid \lambda \in F, (k_{i})_{i \in I} \in V_{s}, \text{ and } \sigma_i( k_{i}) \in F_{\fr{d}} \text{ for all } i \in I\right\}.
    \end{equation*}
    In particular, for all  $u\in   Q(F^{\boldsymbol{\sigma}, \boldsymbol{\rho}})^*$, there is $s\in \mathcal{S}$ and $\gamma\in F^*$ such that 
    $$+_u= +_{\sigma_s \circ \rho_s\circ \varphi_\gamma}.$$ 
\end{theo}

\begin{proof}  
By Lemma \ref{nearfieldauto}, we recall that $\sigma_{i}\circ \rho_{i} \circ \rho_{s}^{-1} \circ \sigma_{s}^{-1}$ being a near-field automorphism is equivalent to $+_{\sigma_{i} \circ \rho_{i}} = +_{\sigma_{s}\circ \rho_{s}}$, for all $i\in I$ and $s\in \mathcal{S}$. Therefore, for all $s\in \mathcal{S}$, and $i\in I_s$ if and only if $+_{\sigma_{i} \circ \rho_{i}} = +_{\sigma_{s}\circ \rho_{s}}$. 
We set 
$$A:=\bigcup_{s\in \mathcal{S}}\left\{\lambda \cdot_{\boldsymbol{\rho}} (k_{i})_{i \in I} \mid \lambda \in F, (k_{i})_{i \in I} \in V_{s}, \text{ and } \sigma_i( k_{i}) \in F_{\fr{d}}, \text{ for all } i \in I\right\}.$$
   Let $\alpha, \beta, \lambda \in F$, $(k_{i})_{i \in I} \in V_{s}$ for some $s \in \mathcal{S}$, so that $\lambda \cdot_{\boldsymbol{\rho}}(k_{i})_{i \in I} \in A$.  Then, we have \begin{align*}
       \alpha \cdot_{\boldsymbol{\rho}}(k_{i})_{i \in I}+_{\boldsymbol{\sigma}} \beta \cdot_{\boldsymbol{\rho}}(k_{i})_{i \in I} & = (\sigma_{i}^{-1}(\sigma_{i}(\rho_{i}(\alpha)k_{i})+\sigma_{i}(\rho_{i}(\beta)k_{i})))_{i \in I} \\
       & = (\sigma_{i}^{-1}((\sigma_{i}(\rho_{i}(\alpha))+\sigma_{i}(\rho_{i}(\beta)))\sigma_{i}(k_{i})))_{i \in I} \\
       & = (\sigma_{i}^{-1}(\sigma_{i}(\rho_{i}(\alpha))+\sigma_{i}(\rho_{i}(\beta)))k_{i})_{i \in I}\\
       & = ((\alpha+_{\sigma_{s}\circ \rho_{s}}\beta)\cdot_{\rho_{i}}k_{i})_{i \in I} \\
       & = (\alpha+_{\sigma_{s}\circ \rho_{s}}\beta)\cdot_{\boldsymbol{\rho}}(k_{i})_{i \in I}
\end{align*}
and so $(k_{i})_{i \in I} \in Q(F^{\boldsymbol{\sigma}, \boldsymbol{\rho}})$. Since $Q(F^{\boldsymbol{\sigma}, \boldsymbol{\rho}})$ is closed under scalar multiplication (see \cite[Hilfssatz 2.2(c)]{Andre} or \cite[Lemma 2.2-2(c)]{DeBruyn}), it follows that $\lambda \cdot_{\boldsymbol{\rho}}(k_{i})_{i \in I} \in Q(F^{\boldsymbol{\sigma}, \boldsymbol{\rho}})$ for any $\lambda \in F$. This concludes the proof of the inclusion $A \subseteq Q(F^{\boldsymbol{\sigma}, \boldsymbol{\rho}})$.\\
For the reverse inclusion, 
since clearly, $\mathbf{0}\in A$, let $(\alpha_{i})_{i \in I} \in Q(F^{\boldsymbol{\sigma}, \boldsymbol{\rho}})$ such that $(\alpha_{i})_{i\in I} \neq \mathbf{0}$. Then, there is $s \in \mathcal{S}$ and $i_{0} \in I_{s}$ such that $\alpha_{i_{0}} \neq 0$. Let $k_{i} = \rho_{i_{0}}^{-1}(\alpha_{i_{0}}^{-1})\cdot_{\rho_{i}}\alpha_{i}$ for all $i \in I$. We now show that for each $k_{i}$ where $i \notin I_{s}$, we have $k_{i}=0$. We have $$(k_{i})_{i \in I} = (\rho_{i_{0}}^{-1}(\alpha_{i_{0}}^{-1})\cdot_{\rho_{i}}\alpha_{i})_{i \in I} = \rho_{i_{0}}^{-1}(\alpha_{i_{0}}^{-1})\cdot_{\boldsymbol{\rho}}(\alpha_{i})_{i \in I}.$$ Since $Q(F^{\boldsymbol{\sigma}, \boldsymbol{\rho}})$ is closed under scalar multiplication (see \cite[Hilfssatz 2.2(c)]{Andre} or \cite[Lemma 2.2-2(c)]{DeBruyn}), we have $(k_{i})_{i \in I} \in Q(F^{\boldsymbol{\sigma}, \boldsymbol{\rho}})$. 
Therefore, given $i \in I$ and $\alpha,\beta \in F$, there is a unique $\gamma \in F$ such that
\begin{align*}
    \alpha \cdot_{\rho_{i}}k_{i}+_{\sigma_{i}}\beta \cdot_{\rho_{i}}k_{i} = \gamma \cdot_{\rho_{i}}k_{i}.
\end{align*}
In particular, $\alpha \cdot_{\rho_{i_{0}}}k_{i_{0}}+_{\sigma_{i_{0}}}\beta \cdot_{\rho_{i_{0}}}k_{i_{0}} = \gamma \cdot_{\rho_{i_{0}}}k_{i_{0}}.$ However, $$k_{i_{0}} = \rho_{i_0}(\rho_{i_{0}}^{-1}(\alpha_{i_{0}}^{-1}))\alpha_{i_{0}} = \alpha_{i_{0}}^{-1}\alpha_{i_{0}} = 1,$$ which means that $\gamma = \alpha +_{\sigma_{i_{0}}\circ \rho_{i_{0}}}\beta= \alpha +_{\sigma_{s}\circ \rho_{s}}\beta$. In turn, we have 
\begin{equation} \label{eqrho}
\alpha \cdot_{\rho_{i}}k_{i}+_{\sigma_{i}}\beta \cdot_{\rho_{i}}k_{i} = (\alpha +_{\sigma_{s}\circ \rho_{s}}\beta) \cdot_{\rho_{i}}k_{i}.
\end{equation}
Therefore, since for any $k_i\neq 0$ we have $k_i \in Q(F^{\sigma_i, \rho_i})$, then for any $k_i \in F^*$, we have  $+_{\sigma_i \circ \rho_i}=+_{\sigma_s \circ \rho_s}$. This implies for all $i\notin I_s$, $k_i=0$ by definition of $\sim$. 
Moreover, Equation (\ref{eqrho}) is equivalent to the equation
$$ 
\sigma_{i}(\rho_{i}(\alpha))\sigma_{i}(k_{i})+\sigma_{i}(\rho_{i}(\beta))\sigma_{i}(k_{i}) = (\sigma_{i}(\rho_{i}(\alpha))+\sigma_{i}(\rho_{i}(\beta)))\sigma_{i}(k_{i}).
$$
Since $\sigma_i \circ \rho_{i}$ is an automorphism, this shows that $\sigma_i(k_i)\in F_{\fr{d}}$, for all $i\in I_s$ and $(\alpha_{i})_{i \in I} \in A$.
This concludes the proof of the reverse inclusion. Furthermore, let $u\in   Q(F^{\boldsymbol{\sigma}, \boldsymbol{\rho}})^*$. From the above we know that there exists $s\in \mathcal{S}$ and $\gamma\in F^*$ such that $u=\gamma \cdot_{\boldsymbol{\rho}} (k_{i})_{i \in I}$ and $(k_{i})_{i \in I} \in V_{s}$, $\sigma_i( k_{i}) \in F_{\fr{d}}$, for all $i \in I$. Then, we conclude that $+_{u}=+_{\sigma_{s}\circ \rho_{s}\circ \varphi_{\gamma}}$ by Lemma \ref{multautnvs}. 
\end{proof}

\begin{defintheo}
\label{pluses}
    Let $(F,+,\cdot)$ be a near-field, $I$ be an index set and $\boldsymbol{\sigma} = (\sigma_i)_{i\in I}$, $\boldsymbol{\rho}=( \rho_i )_{i\in I}$ be tuples of multiplicative automorphisms. We define an equivalence relation $i \overset{..}{\sim} j$ on $I$ as follows: $i \overset{..}{\sim}j$, if there is $\gamma \in F^*$, such that $ \sigma_i  \circ \rho_i \circ \varphi_\gamma \circ (\sigma_{j} \circ \rho_{j})^{-1} $ is a near-field automorphism. Let $\mathcal{T}$ be the full set of representatives with respect to the equivalence relation $\overset{..}{\sim}$ and for any $t \in \mathcal{T}$, we denote $I_{\overset{..}{t}}$ to be the equivalence class of $t$ for this equivalence relation. Let $\boldsymbol{\sigma}_t = ( \sigma_i)_{i \in I_t}$, and $\boldsymbol{\rho}_t=( \rho_i)_{i \in I_t}$. Then $\{\Psi_t(F^{\boldsymbol{\sigma}_{t},\boldsymbol{\rho}_{t}})\}_{t \in \mathcal{T}}$ is a regular decomposition family for $F^{\boldsymbol{\sigma},\boldsymbol{\rho}}$ where $\Psi_t:F^{\boldsymbol{\sigma}_{t},\boldsymbol{\rho}_{t}} \rightarrow F^{\boldsymbol{\sigma},\boldsymbol{\rho}} $ is the map sending $(\alpha_i)_{i\in I_{\overset{..}{t}}}$ to $(\beta_i)_{i \in I}$ where $\beta_i = \alpha_i$, for all $i \in I_t$ and $0$ otherwise. 
\end{defintheo}

\begin{proof} 
For any $\gamma \in F^*$ and $i,j \in I$, by Lemma \ref{multnvs} and \ref{multautnvs}, we know that $+_{\lambda \cdot_{\boldsymbol{\rho}}{\bf e}_{i}} = +_{{\bf e}_{t}}$ is equivalent to $+_{\sigma_{i}\circ \rho_{i}\circ \varphi_{\gamma}}  = +_{\sigma_{t}\circ \rho_{t}}$. This in turn is equivalent to $\sigma_{i} \circ \rho_{i}\circ \varphi_{\gamma} \circ (\sigma_{t}\circ \rho_{t})^{-1}$ being a near-field automorphism, by Lemma \ref{nearfieldauto}.
By the proof of \cite[Theorem 2.4-17]{DeBruyn}, we know that the regular decomposition of $F^{\boldsymbol{\sigma},\boldsymbol{\rho}} $ is $\{ V_t\}_{t\in \mathcal{T}}$ where for all $t\in \mathcal{T}$, $V_t$ is generated by $e_i$ such that $i\in I_{\overset{..}{t}}$. That is, $V_t=\Psi_t(F^{\boldsymbol{\sigma}_{t},\boldsymbol{\rho}_{t}}) $, for all $t\in \mathcal{T}$. This concludes the proof.
\end{proof}

\begin{lemm} 
\label{ismorphic}
   Let $(F,+,\cdot)$ be a near-field, $I$ be an index set and $\boldsymbol{\sigma} = ( \sigma_i )_{i\in I}$, $\boldsymbol{\rho}=( \rho_i )_{i\in I}$ be tuples of multiplicative automorphisms. Then, 
\begin{enumerate}
    \item $F^{\boldsymbol{\sigma},\boldsymbol{\rho}} \simeq F^{\boldsymbol{\theta},\boldsymbol{\operatorname{Id}}}$,
    \item $F^{\boldsymbol{\sigma},\boldsymbol{\rho}} \simeq F^{\boldsymbol{\operatorname{Id}},\boldsymbol{\theta}}$,
    \end{enumerate}
   where $\boldsymbol{\operatorname{Id}} = ( \operatorname{Id})_{i\in I} $ and $\boldsymbol{\theta}= (  \sigma_i \circ \rho_i)_{i\in I} $.
\end{lemm}
\begin{proof}
\begin{enumerate} 
\item We define the map $\phi: F^{\boldsymbol{\sigma},\boldsymbol{\rho}} \rightarrow F^{\boldsymbol{\theta},\boldsymbol{\operatorname{Id}}}$ by sending
${\bf e}_i$ to ${\bf e}_i$.
Extending this map by linearity, we have 
$\phi ( {}^{\boldsymbol{\sigma} \! \!}\sum_{i\in I} \alpha_i \cdot_{\boldsymbol{\rho}} {\bf e}_i)= {}^{\boldsymbol{\theta} \! \!}\sum_{i\in I}  \alpha_i \cdot {\bf e}_i$, where $\alpha_{i} \in F$ for all $i \in I$. 
We show that the map $\phi$ is an isomorphism. By construction, $\phi$ is a bijection. Let $\alpha \in F$ and $\alpha_i, \beta_i\in F$, for all $i \in I$. By Lemma \ref{multautnvs}, we have 
\begin{align*}
    \phi \left({}^{\boldsymbol{\sigma} \! \!}\sum_{i \in I} \alpha_{i} \cdot_{\boldsymbol{\rho}}{\bf e}_{i}  +_{\boldsymbol{\sigma}}  {}^{\boldsymbol{\sigma}\! \!}\sum_{i \in I} \beta_{i} \cdot_{\boldsymbol{\rho}}{\bf e}_{i}\right) & = \phi \left({}^{\boldsymbol{\sigma}\! \! }\sum_{i \in I}(\alpha_{i}+_{\sigma_{i}\circ \rho_{i}}\beta_{i})\cdot_{\boldsymbol{\rho}} {\bf e}_{i} \right) \\
    & = {}^{\boldsymbol{\theta} \! \!}\sum_{i \in I}(\alpha_{i} +_{\sigma_{i}\circ \rho_{i}}\beta_{i})\cdot{\bf e}_{i} \\
    & ={}^{\boldsymbol{\theta} \! \!}\sum_{i \in I}\alpha_{i}\cdot {\bf e}_{i} +_{\boldsymbol{\theta}}{}^{\boldsymbol{\theta} \! \!}\sum_{i\in I}\beta_{i}\cdot{\bf e}_{i} \\
    & = \phi\left({}^{\boldsymbol{\sigma}\! \!}\sum_{i\in I}\alpha_{i}\cdot_{\boldsymbol{\rho}}{\bf e}_{i} \right) +_{\boldsymbol{\theta}} \phi\left({}^{\boldsymbol{\sigma}\! \!}\sum_{i\in I}\beta_{i}\cdot_{\boldsymbol{\rho}}{\bf e}_{i} \right)
\end{align*}
and 
\begin{align*}
    \phi\left(\alpha \cdot_{\boldsymbol{\rho}} {}^{\boldsymbol{\sigma}\! \!}\sum_{i\in I}\alpha_{i}\cdot_{\boldsymbol{\rho}}{\bf e}_{i} \right) &= \phi\left( {}^{\boldsymbol{\sigma}\! \!}\sum_{i\in I}(\alpha\alpha_{i})\cdot_{\boldsymbol{\rho}}{\bf e}_{i} \right) \\
    =& {}^{\boldsymbol{\theta} \! \!}\sum_{i \in I}(\alpha\alpha_{i})\cdot{\bf e}_{i} =\alpha\cdot{}^{\boldsymbol{\theta} \! \!}\sum_{i \in I}\alpha_{i}\cdot{\bf e}_{i} = \alpha\cdot \phi \left({}^{\boldsymbol{\sigma}\! \!}\sum_{i\in I}\alpha_{i}\cdot_{\boldsymbol{\rho}}{\bf e}_{i} \right)
\end{align*}
which shows that $\phi$ is a homomorphism, and this completes the proof.
\item This is a consequence of 1. applying to $\boldsymbol{\sigma}=\boldsymbol{\operatorname{Id}}$ and $\boldsymbol{\rho}=\boldsymbol{\theta}$.
\end{enumerate}
\end{proof}

\begin{defin}
\label{multNVS}
     Let $(F,+,\cdot)$ be a near-field. We say that a near-vector space $V$ over $F$ is {\sf multiplicative } if it is $F$-isomorphic as near-vector spaces to $(F^{\boldsymbol{\sigma},\boldsymbol{\rho}},F)$ for some tuples of multiplicative automorphisms $\boldsymbol{\sigma} = (\sigma_i )_{i\in I}$ and $\boldsymbol{\rho}=( \rho_i )_{i\in I}$. 
    
\end{defin}
\noindent In the following lemma, we explore equivalent conditions that make a near-vector space a multiplicative near-vector space.

\begin{theo} 
\label{mult}
Let $(F,+,\cdot)$ be a near-field $V$ be a near-vector space over $F$. Then, the following assertions are equivalent:
\begin{enumerate} 
\item $(V,F)$ is a multiplicative near-vector space;
\item for all $u \in Q(V)^*$, $+_u=+_\sigma$ for some $\sigma$ multiplicative automorphism of $F$.
\item $V$ is $F$-isomorphic to $F^{\boldsymbol{\theta},\boldsymbol{\operatorname{Id}}}$ as near-vector spaces where $\boldsymbol{\theta}$ is a tuple of multiplicative automorphisms.
\item $V$ is $F$-isomorphic to $F^{\boldsymbol{\operatorname{Id}},\boldsymbol{\theta}}$ as near-vector spaces where $\boldsymbol{\theta}$ is a tuple of multiplicative automorphisms.
\end{enumerate}
\end{theo}
\begin{proof}
$(1) \Rightarrow (2)$ Suppose $(V,F)$ is a multiplicative near-vector space. Then, there is an $F$-isomorphism $\phi: V \rightarrow F^{\boldsymbol{\sigma},\boldsymbol{\rho}}$. Let $u \in Q(V)^{*}$. By Lemma \ref{homoplus}, we have $+_{u} = +_{\phi(u)}$. Furthermore, $\phi(Q(F^{\boldsymbol{\sigma},\boldsymbol{\rho}})) = Q(V)$. Let $\mathcal{S}$ be the set of representatives as defined in Theorem \ref{nvss}. Then, there is $s \in \mathcal{S}$ and $\gamma  \in F^{*}$ such that $+_{\phi(u)} = +_{\sigma_{s}\circ \rho_{s}\circ \varphi_{\gamma}}$ and so $+_{u} = +_{\sigma_{s}\circ \rho_{s}\circ \varphi_\gamma}$.   

\noindent $(2) \Rightarrow (3)$ Suppose that for all $u \in Q(V)^{*}$, $+_{u} = +_{\sigma}$ for some $\sigma$ multiplicative automorphism of $F$. Let $\mathcal{B}$ be a basis of $V$ with $\mathcal{B} \subseteq Q(V)$ and let $\boldsymbol{\sigma}_{\mathcal{B}} = (\sigma_{b})_{b \in \mathcal{B}}$ such that $+_{b} = +_{\sigma_{b}}$ for all $b \in \mathcal{B}$. We define the map $\phi: V \rightarrow F^{\boldsymbol{\sigma_{\mathcal{B}}}, {\bf Id}}$ be sending $b$ to ${\bf e}_{b}$. Extending this map linearly, we have $\phi \left(\sum_{b \in \mathcal{B}}\alpha_{b} \cdot b \right) = {}^{\boldsymbol{\sigma_{\mathcal{B}}}\! \!}\sum_{b \in \mathcal{B}}\alpha_{b}\cdot {\bf e}_{b}$ where $\alpha_{b} \in F$ for all $b \in \mathcal{B}$. By construction, $\phi$ is a bijection. Let $\alpha \in F$ and $\alpha_{b}, \beta_{b} \in F$, for all $b \in \mathcal{B}$. By Lemma \ref{multautnvs}, we have
\begin{align*}
    \phi\left(\sum_{b \in \mathcal{B}} \alpha_{b}\cdot b +\sum_{b \in \mathcal{B}}\beta_{b} \cdot b\right)  = \phi \left(\sum_{b \in \mathcal{B}}(\alpha_{b}+_{\sigma_b}\beta_{b})\cdot b \right) & = {}^{\boldsymbol{\sigma_{\mathcal{B}}}}\sum_{b \in \mathcal{B}}(\alpha_{b} +_{\sigma_b}\beta_{b})\cdot {\bf e}_{b} \\
    & = {}^{\boldsymbol{\sigma_{\mathcal{B}}}\! \!}\sum_{b \in \mathcal{B}}\alpha_{b} \cdot {\bf e}_{b} +_{\boldsymbol{\sigma_{\mathcal{B}}}} {}^{\boldsymbol{\sigma_{\mathcal{B}}}\! \!}\sum_{b \in \mathcal{B}}\beta_{b}\cdot  {\bf e}_{b} \\
    & = \phi \left(\sum_{b \in \mathcal{B}}\alpha_{b}\cdot b \right) +_{\boldsymbol{\sigma_{\mathcal{B}}}} \phi \left(\sum_{b \in \mathcal{B}}\beta_{b}\cdot b \right)
\end{align*}
and 
\begin{align*}
    \phi \left( \alpha \cdot  \sum_{b \in \mathcal{B}}\alpha_{b}\cdot b \right) = \phi \left(\sum_{b \in \mathcal{B}}(\alpha\alpha_{b})\cdot b \right) = {}^{\boldsymbol{\sigma_{\mathcal{B}}}\! \!}\sum_{b \in \mathcal{B}}\alpha \alpha_{b}\cdot {\bf e}_{b} = \alpha \cdot {}^{\boldsymbol{\sigma_{\mathcal{B}}}\! \!}\sum_{b \in \mathcal{B}} \alpha_{b}\cdot {\bf e}_{b} = \alpha \cdot \phi\left(\sum_{b \in \mathcal{B}}\alpha_{b} b\right)
\end{align*}
which shows that $\phi$ is a homomorphism.

\noindent $(3) \Leftrightarrow (4)$ This is by Lemma \ref{ismorphic}. 

\noindent $(4) \Rightarrow (1)$ Suppose $V$ is $F$-isomorphic to $F^{\boldsymbol{\operatorname{Id}},\boldsymbol{\theta}}$ as near-vector spaces where $\boldsymbol{\theta}$ is a tuple of multiplicative automorphisms. $(F^{\boldsymbol{\operatorname{Id}},\boldsymbol{\theta}},F)$ is a special case of $(F^{\boldsymbol{\sigma},\boldsymbol{\rho}},F)$ and so $(V,F)$ is a multiplicative near-vector space.
\end{proof}

\noindent A fundamental question in the study of near-vector spaces is whether infinite direct products preserve the near-vector space structure. We show that this holds under a natural finiteness assumption on the set of near-fields induced by multiplicative automorphisms.
\begin{theo}  
\label{infprod}  
Let $(F, +, \cdot)$ be a near-field such that the set of near-fields induced by its multiplicative automorphisms is finite and $F$ is a finite dimensional vector space over its distributive elements. Let $\{V_k\}_{k \in I}$ be a family of multiplicative near-vector spaces over $F$. Then, the product $\prod_{k \in I} V_k$ is a near-vector space over $F$.  
\end{theo}  

\begin{proof}  
By Theorem \ref{mult}, we may assume that each $V_k$ is of the form $F^{\boldsymbol{\sigma}_k, \mathbf{Id}}$ for some tuple of multiplicative automorphisms $\boldsymbol{\sigma}_k$. We seek to show that $\prod_{k \in I} V_k$ is a near-vector space.  \\
For each $k \in I$, we write $\boldsymbol{\sigma}_k = (\sigma_{k_i})_{i \in J_k}$ for some index set $J_k$, and  we  define $J := \bigoplus_{k \in I} J_k$. Consider the equivalence relation $\overset{..}{\sim}$ on $J$ as given in Definition-Theorem \ref{pluses}, and let $\mathcal{T}$ be a full set of representatives of this equivalence relation. By assumption, $\mathcal{T}$ is finite. Given $t \in \mathcal{T}$, we set $J_{k_t}:=\{ i \in J_k \mid k_i \overset{..}{\sim} t \}$, when $J_{k_t} \neq \emptyset$, $\boldsymbol{\sigma}_{k_t}  :=  (\sigma_{k_i})_{i \in J_k}$, and 
\begin{equation*}
    V_{k_{t}} = \begin{cases} F^{\boldsymbol{\sigma}_{k_t}, \boldsymbol{{\bf Id}}}  & \text{if } J_{k_t} \neq \emptyset; \\
    \{0\} &  \text{otherwise}.
        
    \end{cases}
    \end{equation*}
\\
Thus, we can decompose each $V_k$ as  
\[
V_k = \bigoplus_{t \in \mathcal{T}} V_{k_t} = \prod_{t \in \mathcal{T}} V_{k_t}.\]  
\\
Thus, we obtain  
\[
\prod_{k \in I} V_k = \prod_{k \in I} \prod_{t \in \mathcal{T}} V_{k_t} = \prod_{t \in \mathcal{T}} \prod_{k \in I} V_{k_t}.
\]  
For each $t \in \mathcal{T}$, we claim that $\prod_{k \in I} V_{k_t}$ is a near-vector space. In fact, according to the definition of $\overset{..}{\sim}$, for any non-zero element $v \in \prod_{k \in I} V_{k_t}$, the operation $+_v$ is given by $+_{\sigma_t}$. We show that 
$$Q\left(\prod_{k \in I} V_{k_t} \right)= \left\{\lambda \cdot (l_{k})_{k \in I} \mid \lambda \in F, (l_{k})_{k \in I} \in \prod_{k \in I} V_{k_t}, \text{ and } \sigma_t( l_{k}) \in F_{\fr{d}} \text{ for all } k \in I\right\} .$$
 Let $(\alpha_{k})_{k \in I} \in \prod_{k \in I}V_{k_{t}}$. We know that $((F,+)\cdot)$ over $(F_{\fr{d}},+,\cdot)$ is a finite dimensional vector space. So by Lemma \ref{finitedimensional} and Lemma \ref{divisionring} $2.$, $((F,+_{\sigma_{t}}),\cdot)$ over $(\sigma^{-1}_{t}(F_{\fr{d}}),+_{\sigma_{t}},\cdot)$ is a finite dimensional vector space. Let $\mathfrak{D}$ be a finite basis for $((F,+_{\sigma_{t}}),\cdot)$ over $(\sigma^{-1}_{t}(F_{\fr{d}}),+_{\sigma_{t}},\cdot)$. Then, we have  
$\alpha_{k} = \sum_{\delta_k \in \mathfrak{D}}\delta_k \lambda_{\delta_k}$,
for some $\lambda_{\delta_k} \in \sigma^{-1}_{t}(F_{\fr{d}})$ for all $k \in I$. So, we have that 
\begin{align*}
    (\alpha_{k})_{k \in I} = \left(\sum_{\delta \in \mathfrak{D}}\delta \lambda_{\delta_{k}}\right)_{k \in I} = \sum_{\delta \in \mathfrak{D}}\delta( \lambda_{\delta_{k}})_{k \in I}. 
\end{align*}
Since $\lambda_{\delta_k} \in \sigma^{-1}_{t}(F_{\fr{d}})$, it is clear that $ \sigma_{t}(\lambda_{\delta_{k}}) \in F_{\fr{d}}$ for all $k \in I$. Hence, $\delta(\lambda_{\delta_{k}})_{k \in I} \in Q(\prod_{k \in I}V_{k_t})$ which shows that $Q(\prod_{k \in I}V_{k_t})$ generates $\prod_{k \in I}V_{k_t}$. This proves that it is a near-vector space over $F$.  \\
Since the category of near-vector spaces admits finite coproducts (see Lemma \ref{V^I}), we conclude that $\prod_{k \in I} V_k$ is a near-vector space over $F$.  
\end{proof}

\begin{rem} \label{infprodrem}  
Let $F$ be a scalar group that admits only finitely many addition operations turning it into a near-field and $F$ is a finite dimensional vector space over its distributive elements. From the proof of Theorem \ref{infprod}, it follows that any infinite product of near-vector spaces over $F$ remains a near-vector space. However, determining all possible addition operations on a given scalar group is generally a difficult problem (see \cite{BM2024}).  
\end{rem}  

\section{Properties of multiplicative near-vectors over certain fields}
In this section, we explore the properties of multiplicative automorphisms and multiplicative near-vector spaces over finite fields, the real field, and the complex field. Our aim is to show that the tools developed in \cite{BM2024} allow for explicit computations of these spaces. This framework offers concrete access to a wide class of multiplicative near-vector spaces over these fields.
\subsection{Multiplicative near-vectors over finite fields}
Let \( p \) be a prime number and \( n \in \mathbb{N} \), and consider \( F \), the finite field of order \( p^n \). It is well known that the multiplicative group of \( F \) is cyclic of order \( p^n - 1 \) (see \cite[Section 2]{BM2024}). In the same section, it is established that the group of multiplicative automorphisms of \( F \) is isomorphic to the unit group of \( \mathbb{Z}/(p^n -1)\mathbb{Z} \). These automorphisms are given by  
\[
\phi_{\alpha}: F \to F, \quad x \mapsto x^{\alpha}
\]
for some \( \alpha \in \mathbb{N} \) such that \( \alpha \) is invertible modulo \( p^n - 1 \).  

\noindent It is well known that the full automorphism group of \( (F, +, \cdot) \) is given by 
\[
\operatorname{Aut}(F, +, \cdot) = \langle F_p \rangle,
\]
where \( F_p \) denotes the Frobenius map \( x \mapsto x^p \), and \( \langle F_p \rangle \) is the subgroup it generates in \( (U_{p^n -1}, \odot) \). By \cite[Corollary 1.7]{BM2024}, it follows that the set of all near-fields induced by multiplicative automorphisms \( \mathcal{M}(F)_{\operatorname{aut}}\) of \( F \) satisfies  
\[
\mathcal{M}(F)_{\operatorname{aut}} \simeq U_{p^n -1}/ \langle [p] \rangle.
\]  

\noindent Now, let \( \alpha, \beta \in F \) and consider the multiplicative automorphisms \( \phi_{\alpha}, \phi_{\beta}: F \to F \). By Lemma \ref{nearfieldauto}, the sum operations associated with these automorphisms, \( +_{\phi_{\alpha}} \) and \( +_{\phi_{\beta}} \), coincide if and only if \( \phi_{\alpha} \circ \phi_{\beta}^{-1} \in \langle F_p \rangle \). That is, there exists $m \in \{ 0, \cdots, n\}$ such that
\[
\phi_{\alpha}(\phi_{\beta}^{-1}(x)) = F_p^m(x) \quad \text{for all } x \in F,
\]
which is equivalent to  
\[
x^{\frac{\alpha}{\beta}} = x^{p^m}.
\]
This, in turn, simplifies to the congruence  
\[
\alpha \equiv p^m \beta \mod (p^n - 1).
\]  

\noindent Finally, by Theorem \ref{infprod}, any infinite direct product of multiplicative near-vector spaces over a finite field remains a near-vector space. More generally, by Remark \ref{infprodrem} and \cite[Section 3]{BM2024}, this result extends to arbitrary near-vector spaces over finite fields.

\subsection{Multiplicative near-vectors over real fields}\label{real}

The multiplicative automorphisms of \( \mathbb{R} \) are given by  
\[
\phi_{\alpha}: \mathbb{R}\rightarrow \mathbb{R}, \quad x \mapsto \begin{cases} 
    x^{\alpha} & \text{if } x \geq 0, \\ 
    -(-x)^{\alpha} & \text{if } x < 0 
\end{cases}
\]
for some \( \alpha \in \mathbb{R}^{*} \) (see \cite[Section 2]{BM2024}). It is also noted in \cite[Section 2]{BM2024} that the identity morphism is the only field automorphism of \( (\mathbb{R},+, \cdot) \). Consequently, for any distinct \( \alpha, \beta \in \mathbb{R}^{*} \), we have \( +_{\phi_{\alpha}} \neq +_{\phi_{\beta}} \).  

\noindent We now show that the infinite product  
\[
\prod_{\alpha \in \mathbb{R}^{*}}F^{\boldsymbol{\Phi}, {\bf Id}}
\]
where \( \boldsymbol{\Phi} = (\phi_{\alpha} )_{\alpha \in \mathbb{R}^{*}} \), does not form a near-vector space.  

\noindent To establish this, we first prove that  
\[
Q \left( \prod_{\alpha \in \mathbb{R}^{*}}F^{\boldsymbol{\Phi},{\bf Id}} \right) = \bigcup_{\alpha \in \mathbb{R}^{*}} F {\bf e}_{\alpha}
\]
where \( {\bf e}_{\alpha} = (\delta_{\alpha,\beta})_{\beta \in \mathbb{R}^{*}} \).  

\noindent Let \( (v_{\alpha})_{\alpha \in \mathbb{R}^{*}} \in Q \left( \prod_{\alpha \in \mathbb{R}^{*}}F^{\boldsymbol{\Phi},{\bf Id}} \right) \) and take \( a,b \in \mathbb{R} \). Then, there exists \( \gamma \in \mathbb{R} \) such that  
\[
a(v_{\alpha})_{\alpha \in \mathbb{R}^{*}} +_{\boldsymbol{\Phi}} b(v_{\alpha})_{\alpha \in \mathbb{R}^{*}} = \gamma (v_{\alpha})_{\alpha \in \mathbb{R}^{*}}.
\]
Since addition is defined componentwise, this translates to  
\[
(a+_{\phi_{\alpha}}b)v_{\alpha} = \gamma v_{\alpha} \quad \text{for all } \alpha \in \mathbb{R}^{*}.
\]
For all \( v_{\alpha} \neq 0 \), we obtain  
\[
a+_{\phi_{\alpha}}b = \gamma.
\]
Thus, for any two indices \( \alpha, \beta \) such that \( v_{\alpha} \neq 0 \) and \( v_{\beta} \neq 0 \), it must hold that  
\[
a+_{\phi_{\alpha}}b = a+_{\phi_{\beta}}b.
\]
However, this is impossible since \( +_{\phi_{\alpha}} \neq +_{\phi_{\beta}} \) whenever \( \alpha \neq \beta \). This contradiction implies that at most one \( v_{\alpha} \) is nonzero, meaning that every element of \( Q \left( \prod_{\alpha \in \mathbb{R}^{*}}F^{\boldsymbol{\Phi},{\bf Id}} \right) \) belongs to some \( F{\bf e}_{\alpha} \), proving the inclusion  
\[
Q \left( \prod_{\alpha \in \mathbb{R}^{*}}F^{\boldsymbol{\Phi},{\bf Id}} \right) \subseteq \bigcup_{\alpha \in \mathbb{R}^{*}} F {\bf e}_{\alpha}.
\]
The reverse inclusion is straightforward.  

\noindent Now, suppose that \( Q \left( \prod_{\alpha \in \mathbb{R}^{*}}F^{\boldsymbol{\Phi},{\bf Id}} \right) \) generates \( \prod_{\alpha \in \mathbb{R}^{*}}F^{\boldsymbol{\Phi},{\bf Id}} \) as an additive group. Then, for any \( v \in \prod_{\alpha \in \mathbb{R}^{*}}F^{\boldsymbol{\Phi},{\bf Id}} \) with infinitely many nonzero coordinates, we would have  
\[
v = \sum_{\alpha \in T}\lambda_{\alpha}{\bf e}_{\alpha}, \quad \text{where } T \finsub \mathbb{R}^{*}.
\]
This contradicts our assumption that \( v \) has infinitely many nonzero coordinates, completing the proof.  

\noindent This result holds for any scalar group allowing infinitely many distinct additions, which transforms the scalar group into a near-field. Consequently, near-vector spaces over such scalar groups do not admit infinite products.  

\subsection{Multiplicative near-vectors over complex fields}

Let  
\[
\mathbb{S} = \{s \in \mathbb{C} \mid |s| = 1\}
\]
denote the unit circle when \( F = \mathbb{R} \) or \( F = \mathbb{C} \). For a complex number \( u \), we denote its complex conjugate by \( \overline{u} \).  

\noindent By \cite[Section 2]{BM2024}, the continuous multiplicative automorphisms of \( \mathbb{C}^* \) are of the form:  
\[
\begin{array}{cccl}  
\epsilon_\alpha : &\mathbb{C}^*  & \rightarrow & \mathbb{C}^* \\ 
& z=rs & \mapsto & r^\alpha s
\end{array} 
\quad \text{or} \quad
\begin{array}{cccl}  
\overline{\epsilon_\alpha} : & \mathbb{C}^*  & \rightarrow & \mathbb{C}^* \\ 
& z=rs & \mapsto & r^\alpha\overline{s}
\end{array} 
\]
where \( r \in \mathbb{R}_{>0} \), \( s \in \mathbb{S} \), and \( \alpha \in \mathbb{C} \setminus i \mathbb{R} \).  

\noindent Furthermore, \cite[Section 2]{BM2024} establishes that the only field automorphisms of \( (\mathbb{C},+, \cdot) \) are the identity and complex conjugation. Given two field automorphisms \( \phi, \phi' \) such that $+_\phi= +_{\phi'}$, Lemma \ref{nearfieldauto}, implies that either \( \phi = \phi' \) or \( \phi = \overline{\epsilon_1} \circ \phi' \). Additionally, by \cite[Theorem 2.3]{BM2024}, we have the relations:  
\begin{itemize}
    \item 
If \( \phi' = \epsilon_{\alpha} \), then \( \phi = \overline{\epsilon_{\overline{\alpha}}} \).  
\item If \( \phi' = \overline{\epsilon_{\alpha}} \), then \( \phi = \epsilon_{\overline{\alpha}} \).
\end{itemize}

\noindent As a consequence, just as in the real case, there exist infinitely many distinct additions, preventing all infinite products of multiplicative near-vector spaces over \( \mathbb{C} \) to be near-vector space as proven in the previous section.

\noindent In classical linear algebra, the algebraic closure of \( \mathbb{C} \) simplifies computations involving eigenvalues, eigenvectors, characteristic polynomials, and Jordan forms. However, these tasks become more challenging over \( \mathbb{R} \) since not every polynomial splits into linear factors.  

\noindent To overcome this, vector spaces are often "complexified"— a process that extends a real vector space into a complex vector space while preserving meaningful algebraic structure. This allows us to exploit the algebraic advantages of \( \mathbb{C} \) and later translate the results back to the real setting (see \cite{LA}).  

\noindent A similar notion of complexification is available for multiplicative near-vector spaces.  

\noindent Let \( \alpha \in \mathbb{R}^{*} \), and consider \( \mathbb{R}_{\phi_{\alpha}} \) as a subfield of \( \mathbb{C}_{\epsilon_{\alpha}} \) (or \( \mathbb{C}_{\overline{\epsilon_{\alpha}}} \)). We define the inclusion maps:  
\[
\iota_{\alpha}: \mathbb{R}_{\phi_{\alpha}} \rightarrow \mathbb{C}_{\epsilon_{\alpha}}, \quad x \mapsto x
\]
\[
\text{(resp. } \overline{\iota_{\alpha}}: \mathbb{R}_{\phi_{\alpha}} \rightarrow \mathbb{C}_{\overline{\epsilon_{\alpha}}}, \quad x \mapsto x\text{).}
\]
This forms a simple field extension of degree \( 2 \), with a basis  
\[
\{1, \epsilon_{\alpha}^{-1}(i)\} \quad \text{(resp. } \{ 1, \overline{\epsilon_{\alpha}}^{-1}(i) \} \text{).}
\]
The field \( \mathbb{C}_{\epsilon_{\alpha}} \) is generated over \( \mathbb{R}_{\phi_{\alpha}} \) by \( \epsilon_{\alpha}^{-1}(i) \) (resp. \( \overline{\epsilon_{\alpha}}^{-1}(i) \)), and the minimal polynomial of this extension is  
\[
X^2 +_{\epsilon_{\alpha} } 1 \quad \text{(resp. } X^2 +_{\overline{\epsilon_{\alpha}}} 1\text{).}
\]

\noindent Now, consider a multiplicative real near-vector space \( \mathbb{R}^{\boldsymbol{\sigma},\boldsymbol{\rho}} \), where  
\[
\boldsymbol{\sigma} = ( \phi_{\alpha} )_{\alpha \in T}, \quad \boldsymbol{\rho} = ( \phi_{\beta} )_{\beta \in S}
\]
for some \( T,S \subseteq \mathbb{R}^{*} \).  

\noindent We can define its complexification in at least two possible ways:  
\[
\mathbb{C}^{\boldsymbol{\tilde{\sigma}}, \boldsymbol{\tilde{\rho}}}, \quad \text{where } \boldsymbol{\tilde{\sigma}} = (\epsilon_{\alpha})_{\alpha \in T}, \quad \boldsymbol{\tilde{\rho}} = (\epsilon_{\beta})_{\beta \in S}
\]
or  
\[
\mathbb{C}^{\boldsymbol{\overline{\sigma}}, \boldsymbol{\overline{\rho}}}, \quad \text{where } \boldsymbol{\overline{\sigma}} = (\overline{\epsilon_{\alpha}})_{\alpha \in T}, \quad \boldsymbol{\overline{\rho}} = (\overline{\epsilon_{\beta}})_{\beta \in S}.
\]

\noindent With these constructions of complexification in hand, we can extend those classical linear algebra techniques to the framework of near-vector spaces.

\mb{}

\end{document}